\documentclass[preprint,5p,times,twocolumn]{elsarticle}
\def\subchunk#1{\medbreak\noindent{\it#1}}

\usepackage{graphicx}
\usepackage[misc,geometry]{ifsym}
\usepackage{blindtext}
\usepackage[title]{appendix}
\usepackage{amssymb}
\usepackage{comment}

\usepackage{bbm}

\usepackage{amsmath}

\usepackage{amsthm}

\usepackage{mathrsfs}
\usepackage{tikz}

\usetikzlibrary{arrows,automata}
\usetikzlibrary{decorations.text}

\usepackage{booktabs}

\usepackage{esvect}

\usepackage{array}

\usepackage[utf8]{inputenc}

\usepackage{graphicx}
\usepackage{placeins}

\newtheorem{theorem}{Theorem}
\newtheorem{definition}[theorem]{Definition}
\newtheorem{lemma}[theorem]{Lemma}
\newtheorem{proposition}[theorem]{Proposition}
\newtheorem{corollary}[theorem]{Corollary}

\newtheorem{remark}[theorem]{Remark}
\newtheorem{assumption}[theorem]{Assumption}
\usepackage{lineno}
\usepackage[hidelinks]{hyperref}

\journal{Operations Research Letters}

\begin{document}

\begin{frontmatter}

\title{On Stochastic Auctions in Risk-Averse Electricity Markets With Uncertain Supply}

\author[label1]{Ryan Cory-Wright\corref{cor1}}
\address[label1]{Operations Research Center, Massachusetts Institute of Technology Cambridge, MA, USA}
\cortext[cor1]{Corresponding author}
\ead{ryancw@mit.edu}

\author[add2,add3]{Golbon Zakeri}
\address[add2]{Department of Engineering Science and The Energy Centre, University of Auckland, Auckland, New Zealand}
\ead{g.zakeri@auckland.ac.nz}

\begin{abstract}
This paper studies risk in
a stochastic auction
which facilitates the integration of renewable generation in electricity markets. We model market participants who are risk averse and reflect their risk aversion through coherent risk measures. We uncover a closed form characterization of a risk-averse generator's optimal pre-commitment behaviour for a given real-time policy, both with and without risk trading. 

\end{abstract}


\end{frontmatter}


\section{Introduction}
Renewable power generation is an increasingly attractive investment option for participants in electricity pool markets, as it does not emit carbon and has a marginal cost of zero. Furthermore, investment in intermittent renewable generation is attractive from a regulatory standpoint, as wind and solar generators reduce the expected dispatch cost and do not emit carbon. However, renewable investment increases supply-side uncertainty. This creates difficulties for independent system operators (ISOs) when clearing electricity pool markets, as inflexible coal and nuclear generators may require several hours or more to implement a dispatch and intermittent power output is unknown this far in advance.

When intermittent renewable generators supply a small proportion of electricity, the ISO can efficiently manage deviations from forecasts by procuring suitable amounts of frequency-keeping and reserve generation. However, when intermittent generators supply a larger proportion of electricity, procuring suitable amounts of reserve generation becomes expensive and more efficient grid management strategies are required. Consequently, some market operators employ a two-market strategy, which involves:
\begin{enumerate}
    \item Clearing a pre-commitment market by assuming that renewable generation takes its forecast value.
    \item Letting nature select a realization of uncertainty.
    \item Clearing a real-time market, to balance deviations between renewables' forecast and realised generation output.
\end{enumerate}
This two-market structure allows inflexible generators to implement a dispatch, by providing them with a pre-commitment setpoint. However, the pre-commitment and real-time nodal prices might not converge in expectation, as the expected adjustment cost is not priced when computing the pre-commitment setpoint. As noted in \cite{Zavalaetal}, this price distortion is a market design flaw which may lead to systematic arbitrage opportunities.

More sophisticated uncertainty management strategies comprise modelling intermittent renewable generation as a random variable and computing a pre-commitment setpoint according to the optimal solution of a two-stage stochastic program, which minimizes the expected cost of generation plus adjustment. After uncertainty is realised, a real-time market is subsequently cleared by minimizing the cost of generating electricity plus adjusting to manage fluctuations from forecast renewable generation output.
This market-clearing strategy is known as stochastic dispatch, and has been studied by authors including \cite{BouffardEtal2005a, PZP, ZPBB, UncertainSupply}. Stochastic dispatch almost-surely induces efficiency savings in the long-run \citep[see][]{BirgeAndLoveaux}, because it explicitly prices the expected cost of deviating from a pre-commitment setpoint in the first stage.

In spite of the almost-sure existence of cumulative system savings, we cannot guarantee that all market participants benefit from stochastic dispatch.
Indeed, implementing stochastic dispatch could leave generators out of pocket, and cost recovery is only guaranteed in expectation. 
This raises the question: \textit{What happens if market participants are risk averse?}

This question is also of interest in a different context. One main criticism of stochastic dispatch, is that the system as a whole must have a unified view of the future distribution of outcomes, e.g. all agents must agree with the distribution of wind in the next hour. Allowing for risk aversion offers some flexibility here. Perceiving a different distribution of future outcomes by an agent, can often be equivalently modelled as that agent being risk averse, and equipped with a coherent risk measure. This perspective is often taken in finance, where a martingale measure emerges through a complete market and agents risks are traded \cite{FinanceBook}. We are interested in investigating the outcomes of such markets and we will return to this point in Section \ref{Riskaversesdm}.

\subsection{Contributions and Structure}
The main contributions of this paper are $(a)$ a characterization of the impact of pre-commitment on real-time nodal prices, and its interplay with contracts and $(b)$ a characterization of the impact of risk-aversion on pre-commitment.

The structure of the paper is as follows:
\begin{itemize}
    \item In Section \ref{background}, we briefly outline stochastic dispatch.
    \item In Section \ref{Riskaversesdm}, we study stochastic dispatch in a risk-averse context.
    When generators are endowed with
    coherent risk measures and cannot trade risk,
    the resultant risk-averse competitive equilibrium admits a solution, whenever nodal prices are capped.
    We obtain a closed-form relationship between each generator's pre-commitment and their real-time dispatch, and demonstrate this can results in pre-commitment supply shortfalls. Alternatively, inclusion of Arrow-Debreu securities leads to
    a second risked equilibrium, which itself yields excess pre-commitment. 
\end{itemize}

\section{Background}\label{background}
We start this section by reviewing the stochastic dispatch mechanism (SDM) presented in \cite{ZPBB}. SDM is a mechanism that explicitly models wind supply uncertainty using a probability distribution, and tailors a pre-commitment setpoint to this distribution. Optimal recourse actions are chosen after renewable generation output is revealed.

\subsection*{\rm\textbf{Notation}}
We use lowercase Roman symbols such as $x$ to denote deterministic variables, uppercase Roman symbols such as $X(\omega)$ to denote random variables and lowercase Greek symbols such as $\lambda_n(\omega)$ to denote prices. We use an assortment of sets and indices. We let $\omega \in \Omega$ represent a scenario in our sample space $\Omega$, which we assume to be finite. We let $\mathcal{F}$ be a closed, convex and non-empty set of flows which obey thermal limits, line capacities and the DC load-flow constraints imposed by Kirchhoff's laws. Finally, we let $i$ be the index of a generation unit, $\mathcal{N}$ be the set of all nodes in the network, and $\mathcal{T}(n)$ be the set of all generators located at node $n \in \mathcal{N}$.

We also use an assortment of problem data. We let $c_{i}$ be generator $i$'s marginal generation cost, which we take to be truthfully stated. We let $D_{n}(\omega)$ be the inelastic demand at node $n$ in scenario $\omega$, and $G_{i}(\omega)$ be generator $i$'s production capacity in scenario $\omega$. Finally, we let $r_{u,i}$ (respectively $r_{v,i}$) be generator $i$'s marginal cost of upward (downward) deviation. We require that $r_{u,i}, r_{v,i}>0$ for some generator $i$, as otherwise all generation units are infinitely flexible and can
be dispatched
after uncertainty is realised.

\subsection{The Stochastic Dispatch Mechanism (SDM)}
In SDM, we model renewable generation output by a set of samples from a continuous distribution, which constitutes an ensemble forecast of future uncertainty. Consequently, SDM is a Sample Average Approximation which yields a sequence of pre-commitment setpoints that asymptotically converge to the optimal setpoint as the number of scenarios considered increases \citep[see][]{Shapiro}. We determine SDM's pre-commitment setpoint, $x^{*}$, and the corresponding real-time dispatch policy $X^\star(\omega)$, by solving the following stochastic program: 
\begin{align*}
\text{SLP: } \min \ &   \mathbb{E}[c^\top X(\omega)+r^\top_{u} U(\omega) + r^\top_{v} V(\omega)] & \\
\text{s.t.} \ & \sum_{i \in \mathcal{T}(n)} X_{i}(\omega)+\tau_{n}(F(\omega)) \geq D_{n}(\omega),  & \forall \omega \in \Omega, \ \forall n \in \mathcal{N}, \\
& x + U(\omega) - V(\omega) = X(\omega),  \ &  \forall \omega \in \Omega,  \\
& F(\omega) \in \mathcal{F} , & \forall \omega \in \Omega, \\
& 0 \leq X(\omega) \leq G(\omega), & \forall \omega \in \Omega,\\
& U(\omega), \ V(\omega), \ x \geq 0 , & \forall \omega \in \Omega,
\end{align*}
where $x_{i}$ is generator $i$'s pre-commitment setpoint, the amount generator $i$ prepares to produce before uncertainty is realised, $X_{i}(\omega)$ is generator $i$'s real-time dispatch in scenario $\omega$, $U_{i}(\omega)$ (respectively $V_{i}(\omega)$) is generator $i$'s upward (downward) deviation from its setpoint in scenario $\omega$, $F(\omega)$ is the vector of branch flows through the network in scenario $\omega$, and $\tau_{n}(F(\omega))$ is the net
energy injected from the grid into node $n$ in scenario $\omega$.

After determining the pre-commitment setpoint $x^{*}$, nature selects the scenario $\hat{\omega}$, and the ISO solves the following recourse problem for the real-time dispatch $X(\hat{\omega})$:
 \begin{align*}
 \text{min} \quad  & c^\top X(\hat{\omega})+r^\top_{u} U(\hat{\omega}) + r^\top_{v} V(\hat{\omega}) & \\
 \text{s.t.} \quad & \sum_{i \in \mathcal{T}(n)} X_{i}(\hat{\omega})+\tau_{n}(F(\hat{\omega})) \geq D_{n}(\hat{\omega}), \ \forall n \in \mathcal{N},\quad & [\lambda_{n}(\hat{\omega})],  \\
 & X(\hat{\omega})- U(\hat{\omega}) + V(\hat{\omega}) = x^\star, & [\rho(\hat{\omega})],\\
  & F(\hat{\omega}) \in \mathcal{F}, \quad &  & \\
 & 0 \leq X(\hat{\omega}) \leq G(\hat{\omega}), & & \\
 & U(\hat{\omega}), V(\hat{\omega}) \geq 0, \quad &  &
 \end{align*}
 where $\lambda_n(\hat{\omega})$ is the dual multiplier for the supply-demand balance constraint at node $n$ in scenario $\hat{\omega}$, and $\rho(\hat{\omega})$ is the nonanticipavity dual multiplier in scenario $\hat{\omega}$.  Note that solving the recourse problem (rather than inspecting SLP's recourse policy), is necessary because the first-stage problem considers a Sample-Average-Approximation of uncertainty, and the scenario selected by nature is almost-surely not contained in SLP's ensemble forecast.

After solving both stages, the ISO pays $\lambda_{j(i)}(\hat{\omega})X_{i}(\hat{\omega})$ to generator $i$ and charges $\lambda_{n}(\hat{\omega})D_{n}(\hat{\omega})$ to consumer $n$, where $j(i)$ is the index of the node where generator $i$ is located (but does not compensate each generator's first-stage decision \citep[see][for a justification]{ZPBB}
). The ISO does not incur a penalty for deviating from $f$ to $F(\hat{\omega})$, and therefore is never out of pocket \citep[Proposition 1]{ZPBB}. Moreover, generators recover their fuel and deviation costs in expectation (but not with probability $1$), as both quantities are priced when clearing the first stage \citep[see][]{UncertainSupply}.

\subsection{On Pre-Commitment and Real-Time Nodal Prices}
We now recall some results from convex analysis which we will invoke repeatedly in this paper:

\begin{proposition}\label{nonanticapitivityineq}
Let $x, \hat{x} \geq 0$ be two feasible pre-commitment setpoints, with corresponding optimal recourse dispatches $X(\omega)$, $\hat{X}(\omega)$. Then, the optimal second-stage nonanticipavity dual multipliers, $\rho(\omega), \hat{\rho}(\omega)$ obey the following relationship:
\begin{align*}
    \mathbb{P}\bigg(\left\langle x-\hat{x}, \rho(\omega)-\hat{\rho}(\omega) \right\rangle \geq 0\bigg)=1.
\end{align*}
\end{proposition}
\begin{proof}
This follows directly from the observation that the subgradient of the first-stage dispatch problem with respect to $x$ in scenario $\omega$, $\rho(\omega)$, is a maximal monotone operator \citep[see][]{RockafellarSubdiff}.
\end{proof}

\begin{corollary}\label{corr:priceprecommitmentrelationship}
Let generator $i$'s dispatch under SDM in scenario $\omega$ be $0 < X_i(\omega) < G_i$. Then, for sufficiently small changes $\delta$ in $x_i$ (such that the same inequality constraints remain binding in the real-time problem), we have that
\begin{itemize}
    \item if $x_{i,new}=x_i +\delta>x_i$ then either $\Delta \lambda_{j(i),new}(\omega)=0$ or $\Delta \lambda_{j(i),new}(\omega)=r_{u,i}+r_{v,i}$, and,
    \item if $x_{i,new}=x_i +\delta<x_i$ then either $\Delta \lambda_{j(i),new}(\omega)=0$ or $\Delta \lambda_{j(i),new}(\omega)=-r_{u,i}-r_{v,i}$.
\end{itemize}
\end{corollary}
\begin{proof}
Observe that if $0 < X_i(\omega) < G_i$ then the real-time dispatch problem's KKT condition with respect to $X_i(\omega)$ is:
\begin{align*}
    \lambda_{j(i)}(\omega)+\rho_i(\omega)=c_i,
\end{align*}
meaning we have that $\Delta\lambda_{j(i)}(\omega)+\Delta\rho_i(\omega)=0$.

Moreover, the KKT conditions with respect to $U_i(\omega)$, $V_i(\omega)$ imply that $-r_{u,i} \leq \rho_i(\omega) \leq r_{v,i}$. Therefore, there exists some optimal basis where $\rho_i(\omega)=r_{v,i}$ or $\rho_i(\omega)=-r_{u,i}$ for each generator $i$, without loss of generality. Invoking Proposition \ref{nonanticapitivityineq} then yields the result, where we use the continuity of the optimal primal solution in the problem data \citep[see][]{NemirovskiLectureNotes} to ensure that the same primal constraints remain binding.
\end{proof}
\begin{remark}\label{monotoneoperatorrtp}
The real-time nodal price $\lambda$ is a maximal monotone operator with respect to the realised net demand $D$. Therefore, for two given real-time demand vectors $D, \hat{D}$ the real-time prices obey the relationship $\left \langle \hat{D}-D, \hat{\lambda}-\lambda \right \rangle \geq 0$, and, since generator supply margins are nondecreasing functions of the real-time price, deterministic generators prefer low wind (higher net demand) periods. That is, if generators are risk-averse then their risk-aversion causes them to place additional emphasis on high-wind periods.
\end{remark}

\section{Two Risk-Averse Stochastic Auctions}\label{Riskaversesdm}
In this section we study the outcome of agent interactions when agents are risk averse with no
opportunity to trade risk, and uncover the inefficiencies that transpire as a result. We then move on to completing the
risk trade market, and demonstrate that much like in finance \cite{FinanceBook}, when coherent  risk measures are used and in presence of a complete market, a martingale measure emerges, leading to an equivalent social planning model that embeds participants' risks, and efficiency is restored.

We restrict our attention to law-invariant coherent risk measures because of their natural dual representability  \citep[see][]{Kusuoka2001}. Our analysis resembles the analysis of risk-trading conducted by \cite{RalphAndSmeers} although we restrict our attention to energy-only markets. This restriction allows us to exploit the properties of risk-averse newsvendors established by \citet{MultiProductNewsVendor}.


\begin{definition}\rm
A coherent risk measure, $\rho:\mathcal{Z} \mapsto \mathbb{R}$, as defined in \cite{Artzneretal}, is a function which measures the risk-adjusted disbenefit of a random variable $Z$.

In this paper, we will work with the dual representation of a coherent risk measure, i.e. there exists a convex set of measures $\mathcal{D}$ where $\rho$ can be defined by:
\begin{gather*}
    \rho(Z)=\max_{\mu \in \mathcal{D}} \mathbb{E}_{\mu}[Z].
\end{gather*}
By Kusuoka's Theorem \citep[see][]{Kusuoka2001} each coherent risk measure can be represented in the following mean-risk form for some risk coefficient $\kappa$ and some risk set $\mathcal{D}$:
\begin{align*}
    \rho(Z)&
    =-\mathbb{E}[Z]+\kappa \sup_{\mu \in D} \int_{\beta=0}^{\beta=1}\dfrac{1}{\beta}q_{\beta}[Z] \mu d\beta\\
    & =-\kappa \sup_{\mu \in D} \int_{\beta=0}^{\beta=1}\mathrm{CVaR}_{\beta}[Z] \mu d\beta,
\end{align*}
where
\begin{align*}
    q_{\beta}[Z]=\min_{\eta \in \mathbb{R}}\mathbb{E}[\max((1-\beta)(\eta-Z), \beta(Z-\eta))]
 \end{align*} is the weighted mean-deviation from the $\beta$th quantile, $\kappa$ is a constant which prices risk by balancing the desirability of maximizing $\mathbb{E}[Z]$ with the undesirability of fluctuations towards the left tail of $Z$, and $\mathcal{D}$ is a convex subset of probability measures. 
 \end{definition}
 As observed by \citet{PhilpottFerrisWets16}, whenever the sample space $\Omega$ is finite, at least one of the worst-case probability measures in $\mathcal{D}$ is an extreme point of $\mathcal{D}$. Consequently, in a Sample Average Approximation (SAA) setting, $\rho(Z)$ is equal to the optimal value of the following linear program:
\begin{align*}
    \rho(Z)= \min \quad & \theta \ \text{s.t. }\ \theta \geq \sum_{\omega}\mathbb{P}_{m}(\omega) Z({\omega}), \ \forall m,
\end{align*}
where $\mathbb{P}_m$ is the measure which corresponds to the $m$th extreme point of $\mathcal{D}$. Moreover, if $\mathcal{D}$ is polyhedral then the cardinality of $m$ is finite and the above problem can be solved in a tractable fashion via linear programming, and otherwise it can be solved via a cutting-plane method, such as a Bundle Method \citep[see][]{lemarechal1995new}.

\subsection{Risk-Averse SDM Without Risk Trading
}
\label{sssec:RiskAversionNOAD}

SDM is equivalent to a system where risk neutral price-taking participants optimize their return against the so called Walrasian auctioneer's announced prices \citep[see][]{ZPBB}. Below, we investigate a risk-averse equilibrium, wherein each generation agent $i$ is endowed with a coherent risk measure $\rho_{i}$, with a view to show that the risk-averse equilibrium always admits at least one solution. Observe that since the ISO and the market clearing agent do not have first-stage actions, they perform the same action under any coherent risk measure. Therefore, we treat both agents as risk-neutral without loss of generality. We now state the ISO's problem (PISO) and the market clearing problem (MC) below:
\begin{gather*}
\text{PISO}(\omega):\ \max \sum_{n}\lambda_{n}(\omega)\tau_{n}(F(\omega)) \ \text{s.t.\thinspace\ } F(\omega) \in \mathcal{F},
\end{gather*}
\begin{gather*}
\text{MC}(\omega):\ \max -\sum_{n}\lambda_{n}(\omega)\big(\sum_{i \in T(n)}X_{i}(\omega)+\tau_{n}(F(\omega))-D_{n}(\omega)\big) \\
\text{s.t.\thinspace\ } \lambda_{n}(\omega) \geq 0, \ \forall n \in \mathcal{N},
\end{gather*}
Notably, however, each generation agent $i$ makes a pre-commitment decision $x_i$ in the first-stage, which may be affected by risk-aversion. Therefore, we reflect each generation agent $i$'s attitude towards risk via the risk measure $\rho_i(\cdot)$. Given this risk measure and prices $\lambda_{i}(\omega)$ in each scenario $\omega$, each generation agent $i$ maximizes its risk-adjusted expected profit by determining the actions $(x, X(\omega), U(\omega), V(\omega))$ which solve the following stochastic optimization problem:
\begin{align*}
\text{RAP}(i): \max \quad & \rho_{i}\Bigl( (\lambda_{j(i)}(\omega)-c_{i}) X_{i}(\omega )-r_{u,i} U_{i}(\omega) -r_{v,i} V_{i}(\omega ) \Bigr) \\
\text{s.t.} \quad & x_{i} + U_{i}(\omega) - V_{i}(\omega) = X_{i}(\omega),  \quad \forall \omega \in \Omega, \\
& 0 \leq X_{i}(\omega ) \leq G_{i}(\omega), \quad \forall \omega \in \Omega,\\
& U_{i}(\omega ),V_{i}(\omega ), x_{i} \geq 0, \quad \forall \omega \in \Omega.
\end{align*}

Observe that generation agent $i$ almost-surely recovers its risk-adjusted costs in the long-run, since it can choose the action $(x, X, U, V)=(0, 0, 0, 0)$ and earn a certain payoff of $0$.

The collection of the problems PISO($\omega$), RAP($i$) and $MC(\omega)$ then defines a risk-averse competitive equilibrium, which we refer to as RAEQ. Our subsequent analysis assumes that RAEQ admits a solution, and therefore requires an existence result. To obtain this
result, we require the following intermediate lemma:

\begin{lemma}
\label{gensetisbounded}
Let generation agent $i$ be a risk-averse price-taking generation agent endowed with the coherent risk measure $\rho_i$. Then agent $i$'s optimization problem, RAP($i$), has a closed, convex and bounded strategy set.
\end{lemma}

\begin{proof}
The constraint $0 \leq X_{i}(\omega) \leq G_{i}(\omega)$ implies generator $i$'s optimal action, $x^\star_i$, satisfies the inequality $0 \leq x^\star_i \leq \max_{\omega} \{G_{i}(\omega)\}$,  as otherwise it can decrease its setpoint from $x_i$ to $G_i$ and almost-surely earn an additional profit of $(x_i-G_i)r_{v,i}$, which implies that if $x_i >G_i$ is an optimal setpoint then $G_i$ is also an optimal setpoint. Therefore, we can introduce the constraint ${0 \leq x^\star_i \leq \max_{\omega} \{G_{i}(\omega)\}}$ into the problem RAP($i$), without loss of optimality. The restrictions on $X$ and $x$ then imply that ${0 \leq U_{i}(\omega), V_{i}(\omega) \leq G_{i}(\omega)}$, since $U_{i}(\omega)=\max(X_{i}(\omega)-x_{i}, 0)$ and ${V_{i}(\omega)=\max(x_{i}-X_{i}(\omega), 0)}$. Therefore, the strategy space RAP($i$) is bounded.

The strategy space is also closed and convex, because it is defined by the intersection of
linear inequality constraints.
\end{proof}

We also require the following assumption:
\begin{assumption}\rm
\label{VOLLAssumption}
The optimal choice of dual price $\lambda_{n}(\omega)$ is bounded from above by the Value of Lost Load, or VOLL, for all nodes $n$ and all scenarios $\omega$.
\end{assumption}

Assumption \ref{VOLLAssumption} is common in power system applications; for instance, the New Zealand Electricity Market has a price cap of VOLL=$\$20,000$ per MWh, meaning consumers are willing to curtail their load at a marginal price of $\$20,000$ per MWh in the short-run (see \cite{Stoft} for a general theory). Moreover, \citep[Lemma $1$]{ZPBB} establishes that for a fixed pre-commitment setpoint $x$ and set of real-time demand realisations, there exists some price cap such that Assumption \ref{VOLLAssumption} holds everywhere except a set of measure $0$.

Combining Lemma \ref{gensetisbounded} and Assumption \ref{VOLLAssumption} then yields:
\begin{theorem} \label{raeqadmitsasolution}
Suppose that Assumption \ref{VOLLAssumption} holds. Then, RAEQ admits a solution.
\end{theorem}

\subchunk{Proof.}
To show this result, we follow the steps of Rosen's theorem \citep[see][]{Rosen} in arguing that the following three conditions hold:
\begin{enumerate}
    \item Each participant's strategy set is non-empty.
    \item Each participant's strategy set is closed, convex and bounded.
    \item Each participant's payoff function is concave in her strategy,
    and continuous in all other participants' strategies.
\end{enumerate}

The first statement holds, as each participant can choose the feasible action of setting all their decision variables to $0$, and therefore all participants have non-empty strategy sets.

To show that the second statement holds, we consider each class of agent separately. First, the problem PISO($\omega$)'s strategy space is the closed, convex and bounded set $\mathcal{F}$, which implies that the second statement holds for PISO($\omega$). Second, by Lemma \ref{gensetisbounded}, each generation agent $i$'s strategy set is closed, convex and bounded. Finally, by Assumption \ref{VOLLAssumption}, the market clearing agent's strategy space is a bounded set which can be seen to be closed and convex by inspection. Therefore, the second statement holds for all participants in RAEQ.

The third statement holds for all generation agents, as their decision variables are continuous and they are endowed with coherent risk measures, i.e., convex risk measures where positive homogeneity also holds. Consequently, their payoff functions are concave with respect to maximization. Similarly, the third statement holds for the ISO and market-clearing agents, as they solve wait-and-see optimization problems by choosing continuous decision variables from convex strategy sets.\qed


\begin{remark}\rm
{\rm \textbf{Existence does not imply uniqueness.}} \\
Theorem \ref{raeqadmitsasolution} shows that, with a price cap of VOLL, there exists a set of prices which clear the market when the participants are risk-averse. However, Theorem \ref{raeqadmitsasolution} does not imply that this set of prices is unique. Indeed, \cite{GerardEtAl} provides examples of risk-averse energy-only pool markets which admit multiple equilibria.
\end{remark}

We
now consider RAP($i$)'s first-order optimality condition, with a view to obtaining insight into the relationship between $x^\star$ and $X^\star(\omega)$. To do so, we
assume that $\Omega$ represents the true distribution of uncertainty. Consequently, the below results hold for the true distribution of uncertainty, while solutions to RAEQ constitute SAA estimators of the solution for the true distribution. However, SAA estimators for variational inequalities converge exponentially fast as we increase the sample size \citep[see][]{xu2010sample}. Therefore, the below results hold for RAEQ, in the limit of large samples.

We also need to clarify our understanding of the remuneration process. To see this, consider the auctioneer's price-setting problem MC$(\omega)$ in scenario $\omega$, and assume that the optimal choice of dual price is $VOLL>\lambda^\star_n(\omega)>0$ for some node $n$. Then, the corresponding DC-load-flow constraint must be met exactly, because otherwise the unique optimal choice of nodal price is VOLL. Therefore, we have that any $\lambda_n(\omega) \in [0, VOLL]$ is an optimal choice of dual price at this node, with all such choices providing the auctioneer with a payoff of $0$. That is, the remuneration scheme suggested by RAEQ provides highly degenerate dual prices. Consequently, we assume that participants are dispatched and remunerated in the same manner as SLP, although they may be risk-averse when making their pre-commitment decision. In this context, each generation agent $i$ solves a risk-averse newsvendor problem to determine their pre-commitment behaviour.

This observation allows us to characterize the impact of a generator's risk-aversion on their pre-commitment behaviour in the following proposition:

\begin{proposition}
\label{GeneralRepresentationNoAD}
Let generator $i$ be risk-averse and endowed with the coherent risk measure $\rho:\mathcal{Z} \mapsto \mathbb{R}$, which has the following Kusuoka representation:
\begin{align*}
    \rho[Z]=-\mathbb{E}[Z]+\kappa \sup_{\mu \in D} \int_{\beta=0}^{\beta=1}\dfrac{1}{\beta}q_{\beta}[Z] \mu d\beta.
\end{align*}
Then, for a given set of second stage dispatches $X^\star_{i}(\omega)$, generator $i$ makes the pre-commitment decision $x^\star_{i}$, where:
\begin{align*}
    x^\star_i=F^{-1}_{X^\star_{i}(\omega)}\Big(\dfrac{r_{u,i}}{(r_{u,i}+r_{v,i}){(1+\kappa(1-\bar{\beta}))}} \Big),\\
    \bar{\beta}=\int_{0}^{1}\mu^{RN}\beta d\beta; \ \kappa \in [0, \dfrac{1}{\bar{\beta}}],
\end{align*}
and $F^{-1}(\cdot)$ denotes the pseudoinverse CDF of the probability distribution of $X^\star_{i}(\omega)$.
\end{proposition}

\subchunk{Proof.}
See Appendix \ref{sssec:proofnoad}.
\qed

Proposition \ref{GeneralRepresentationNoAD} indicates that increasing generator $i$'s risk coefficient ${\kappa (1-\bar{\beta})}$ from its risk-neutral level ${\kappa (1-\bar{\beta})=1 (1-1) = 0}$ results in generator $i$ decreasing their pre-commitment setpoint from its risk-neutral level. We formalize this observation in the following corollary to Proposition \ref{GeneralRepresentationNoAD}.
\begin{corollary}\label{rnvsranoadb}
Let generator $i$ be endowed with the coherent risk measure $\rho_i$. Then, for a given set of second stage dispatches $X^\star_{i}(\omega)$, generator $i$ makes the pre-commitment decision $x^\star_i(X_i^\star(\omega))$, which is such that:
\begin{align*}
    x^\star_i(X_i^\star(\omega)) \leq x^\star_{i,RN}(X_i^\star(\omega)),
\end{align*}
where $x^\star_{i,RN}(X_i^\star(\omega))$ is generator $i$'s risk-neutral pre-commitment decision for the second-stage dispatches $X_i^\star(\omega)$.
\end{corollary}
\subchunk{Proof.}
The proof follows from a simple application of Proposition \ref{GeneralRepresentationNoAD}.
\qed

\begin{remark}
We note that Corollary \ref{rnvsranoadb} assumes a given set of $X_i^\star(\omega)$, while that these real time dispatches may change with a change in risk aversion. This leaves open the possibility that the {\em equilibrium} pre-commitment then may not have decreased from the starting pre-commitment. However, we draw the reader's attention to a {\em symmetric duopoly}, where demand should be met by two symmetric generators. In this case, the result is a clear decrease in pre-commitment, even in the equilibrium of the game when both firms change their risk aversion level.
\end{remark}

The analysis in the previous sections indicates that modifying the total pre-commitment magnitude impacts the payoffs to the market participants. Consequently, a pertinent question is ``what is the impact of generator risk-aversion on the generator's expected payoff?''. We provide a lower bound on this quantity in the following proposition:

\begin{proposition}\label{expectedriskaverseprofit1}
Let generator $i$'s risk-aversion be represented by the risk measure $\rho_i$, which has a Kusuoka representation such that $\bar{\beta}_i:=\int_0^1 \mu^{RN} \beta_i d\beta_i$, ${\kappa_i \in [ 0 , \frac{1}{\bar{\beta}_i} ]}$, and combine these quantities by defining $\alpha_i:=\frac{1}{1+\kappa_i(1-\bar{\beta}_i)}$. Then generator $i$'s expected profit is at least $(1-\alpha_i)r_{u,i}x_i^\star$.
\end{proposition}

\begin{proof}
Proposition \ref{GeneralRepresentationNoAD} shows that the quantity $\alpha_i:=\frac{1}{\kappa_i-\kappa_i \bar{\beta_i}}$ summarizes the relationship between generator $i$'s pre-commitment and its production, since $x_i^\star$ is a $\frac{\alpha_i r_{u,i}}{r_{u,i}+r_{v,i}}$ quantile of the distribution of $X_i^\star(\omega)$. Therefore, $X_i^\star(\omega) \leq x_i^\star$ with probability $\frac{\alpha_i r_{u,i}}{r_{u,i}+r_{v,i}}$. Moreover, it follows from \citep[Corollary 1]{ZPBB} that generator $i$ receives a payoff of at least $-r_{v,i}x_i^\star$ with probability $\frac{\alpha_i r_{u,i}}{r_{u,i}+r_{v,i}}$, and at least $r_{u,i}x_i^\star$ with probability $\frac{(1-\alpha_i) r_{u,i}+r_{v,i}}{r_{u,i}+r_{v,i}}$. Computing the expected payoff then yields the result.
\end{proof}

Proposition \ref{expectedriskaverseprofit1} suggests that generator risk-aversion results in a lower pre-commitment magnitude than the optimal risk-neutral setpoint, which reduces cumulative system welfare. Moreover, Proposition \ref{expectedriskaverseprofit1} indicates that marginal generators may have an incentive to behave in a risk-averse manner, as the expected profit of risk-neutral marginal generators who are never dispatched at their output capacity is $0$ (this follows from Proposition $1$), and the expected profit of risk-averse marginal generators is bounded from below by a strictly positive quantity.

To see that this situation can also arise in SDM, observe that generators can express their risk-aversion by inflating the relative magnitude of $r_{v,i}$, their marginal cost of deviating downward, in order that the auctioneer dispatches them at a lower pre-commitment setpoint. Indeed, a recent numerical study \cite{KazempourPinson} confirms our finding, by demonstrating that in a two-market stochastic equilibrium where generators are endowed with the $\mathrm{CVaR}$ risk criterion, generators prefer to pre-commit less generation when they are more risk-averse

Fortunately, \citet{RalphAndSmeers} provide a framework for extending SDM to cope with risk-aversion: introducing an auxiliary financial market wherein generators and the ISO can trade risk. If generators and the ISO are endowed with intersecting risk sets, then trading Arrow-Debreu securities causes each participant's effective risk-aversion to decrease to the least risk-averse participant's risk-aversion, leaving only residual risk. In this case, each generator's pre-commitment decision is equivalent to the decision made by a risk-averse system optimizer who uses the \emph{least} risk averse agent's risk set as its own.


\subsection{Risk-averse SDM With Risk Trading
}
\label{sssec:RiskAversionWithAD}
In this section, we extend our preceding analysis to consider a stochastic energy-only market where participants trade Arrow-Debreu securities on an exchange. We begin by defining the market clearing problem.

We require the following definition:
\begin{definition}
An Arrow-Debreu security is a contract which charges the price $\pi({\omega})$ to receive a payoff of $1$ in scenario $\omega$. We let $W_{i}(\omega)$ denote the bundle of Arrow-Debreu securities held by agent $i$ \citep[see][]{RalphAndSmeers}.
\end{definition}

We also require the following notation:
\begin{itemize}
\item $\theta_{i}$ is generator $i$'s risk-adjusted payoff.
\item $\theta_{k}$ is the ISO's risk-adjusted payoff.
\item $W_{k}(\omega)$ is the quantity of Arrow-Debreu securities purchased by the ISO in scenario $\omega$.
\item $\mathbb{P}_{im}(\omega)$ is the probability measure corresponding to the $m$th extreme point of generator $i$'s risk set.
\item $\mathbb{P}_{km}(\omega)$ is the probability measure corresponding to the $m$th extreme point of the ISO's risk set.
\end{itemize}

Assume that each generator submits the same offers as in SDM, that all generators and the ISO submit their risk sets before the market is cleared, and that the intersection of the participants' risk sets is non-empty. Then, clearing the market is equivalent to minimizing cumulative risk-adjusted disutility \cite{RalphAndSmeers}, i.e., solving the following risk-averse stochastic program:

\begin{align*}
\text{RASLP:}\quad & \min \sum_{i} \theta_{i} + \theta_{k}\\
\text{s.t. }\quad & \theta_{i} \geq \sum_{\omega}\mathbb{P}_{im}(\omega)\bigg(c_{i}X_{i}(\omega)+r_{u,i}U_{i}(\omega) \\
& \qquad \qquad +r_{v,i}V_{i}(\omega)-W_{i}(\omega)\bigg), \forall i,  \forall m,\\
& \theta_{k}+ \sum_{\omega}\mathbb{P}_{km}(\omega)W_{k}(\omega) \geq 0,  \forall m, \\
& \sum_{i \in \mathcal{T}(n)} X_{i}(\omega)+\tau_{n}(F(\omega)) \geq D_{n}(\omega),  \forall \omega \in \Omega, \\
& \sum_{i}W_{i}(\omega)+W_{k}(\omega)=0, \forall \omega \in \Omega, [\pi(\omega)], \\
& x + U(\omega) - V(\omega) = X(\omega),   \forall \omega \in \Omega,  \\
& F(\omega) \in \mathcal{F} , \quad \forall \omega \in \Omega, \\
& 0 \leq X(\omega) \leq G(\omega) , \ U(\omega), V(\omega), x \geq 0 , \quad \forall \omega \in \Omega,
\end{align*}
where we
enumerate the extreme points of each generator's risk set in order to express the market-clearing problem as a single linear program. Note that RASLP may not be a linear program of finite size, as participants may reflect their risk-aversion via non-polyhedral risk sets. Nonetheless, RASLP can be solved in a tractable fashion, by $(1)$ observing that it corresponds to maximizing risk-adjusted social welfare under the least risk-averse participants risk measure, and $(2)$ solving this risk-averse market clearing problem using either (a) an interior point method if the least risk-averse participant has a polyhedral risk measure or (b) the Bundle Method \citep[see][]{lemarechal1995new} if the least risk-averse participant has a non-polyhedral risk measure.

After solving RASLP, participants are remunerated for their dispatch in the same manner as SDM, and participants are remunerated with the term $W_{i}(\hat{\omega})-\sum_{\omega}\pi(\omega)W_{i}(\omega)$ in scenario $\hat{\omega}$ for their financial instruments, as per \cite{RalphAndSmeers}.

As noted by \cite{RalphAndSmeers}, the dual prices of the Arrow-Debreu securities, $\pi(\omega)$, correspond to the system optimizer's risk-adjusted probability measure. The equivalence between dual prices and the worst-case probability measure allows us to rewrite the system optimization objective function as:
\begin{align*}
    \min_{x,u,v} \rho(c^\top X(\omega)+r^\top_{u} U(\omega) + r^\top_{v} V(\omega)),
\end{align*}
where $\rho$ is a coherent risk measure with risk set $\mathcal{D}$. If ${\mathcal{D}=\{\mathbb{P}(\omega)\}}$ then ($1$) there exists a risk-neutral agent which absorbs all risk in the market, ($2$) the Arrow-Debreu securities are priced at $\mathbb{P}(\omega)$, and ($3$) there is no residual system risk \citep[see][]{RalphAndSmeers}.

Interestingly, unlike the risk-averse competitive equilibrium studied in the previous section, it is straightforward to elicit verifiable conditions for existence and uniqueness of a risk-averse competitive equilibrium in the presence of risk trading. Existence can be verified by solving the system optimization problem. Moreover, if $\mathcal{F}$ is a polyhedral set, uniqueness can be verified by
following \citep[Exercise $3.9$]{IntroLinearOptBT}.

Our main interest in this paper is determining the impact of the existence of financial instruments on the pre-commitment setpoint. Consequently, we change our perspective and assume that $\Omega$ represents the true distribution of uncertainty. Strictly speaking, the optimal solution to RASLP constitutes an SAA estimator of the optimal solution for the true distribution. However, SAA estimators are known to converge almost surely to the optimal solution for the underlying distribution \citep[see][]{Shapiro}. Therefore, the below results hold almost surely true for solutions to RASLP, wherever the sample of the underlying distribution is sufficiently rich.

We
now study the system optimization's first-order optimality condition with respect to each generator $i$. As we are considering a system optimization problem rather than an individual generator's problem, our objective is risk-adjusted expected fuel cost minimization rather than risk-adjusted expected profit maximization, and we are risk-averse to scenarios with high fuel plus deviation costs rather than scenarios with low nodal prices. This observation leads to the following proposition:

\begin{proposition}
\label{GeneralRepresentationWithAD}
Suppose that the system is risk-averse and endowed with the law invariant coherent risk measure $\rho:\mathcal{Z} \mapsto \mathbb{R}$, which has the Kusuoka representation:
\begin{align*}
    \rho[Z]=-\mathbb{E}[Z]+\kappa \sup_{\mu \in D} \int_{\beta=0}^{\beta=1}\frac{1}{\beta}q_{\beta}[Z] \mu d\beta.
\end{align*}
Then, for a given dispatch policy $X^\star(\omega)$, each generator makes the pre-commitment decision to produce $x^\star_{i}$, where:
\begin{align*}
    x^\star_i=F^{-1}_{X^\star_{i}(\omega)}\Big(\frac{r_{u,i}+(r_{u,i}+r_{v,i})(\kappa-\kappa\bar{\beta})}{(r_{u,i}+r_{v,i}){(1+\kappa(1-\bar{\beta}))}} \Big),\\
    \bar{\beta}=\int_{0}^{1}\mu^{RN}\beta d\beta, \ \kappa \in [0, \dfrac{1}{\bar{\beta}}],
\end{align*}
and $F^{-1}(\cdot)$ denotes the pseudoinverse CDF of the probability distribution of $X^\star_{i}(\omega)$.
\end{proposition}


\subchunk{Proof.}
See Appendix \ref{sssec:proofwithad}.
\qed

We remind the reader that the dispatch policy $X^\star(\omega)$ obtained from RASLP and used in Proposition \ref{GeneralRepresentationWithAD} is not necessarily the same dispatch policy as that obtained from RAEQ and used in Proposition \ref{GeneralRepresentationNoAD}. In particular, both dispatch policies are functions of $(1)$ the problem data and $(2)$ their (respective and possibly different) pre-commitment setpoints.

Proposition \ref{GeneralRepresentationWithAD} has the following interpretation: although the real-time price $\lambda(\omega)$ is a maximally monotone operator with respect to realised demand, meaning risk-averse generators place additional emphasis on high-wind scenarios and reduce their pre-commitment from a risk-neutral setpoint, Arrow-Debreu securities re-align generators incentives, causing them to view high-wind scenarios favourably, and increase their pre-commitment magnitude from its risk-neutral setpoint. We formalize this observation in the following corollary:

\begin{corollary}\label{rnvsrawithab}
Let the system be endowed with the coherent risk measure $\rho$. Then, for a given second stage dispatch policy $X^\star(\omega)$, each generator makes the pre-commitment decision $x^\star_i(X_i^\star(\omega))$, which is bounded from below by the following expression:
\begin{align*}
    x^\star_i(X_i^\star(\omega)) \geq x^\star_{i,RN}(X_i^\star(\omega)),
\end{align*}
where $x^\star_{i,RN}(X_i^\star(\omega))$ is generator $i$'s risk-neutral pre-commitment decision.
\end{corollary}


We remind the reader that the second-stage dispatches from SLP and RASLP are distinct, meaning we cannot make a direct comparison between $x_i^\star$ and $x^\star_{i, RN}$. However, recalling that the optimal real-time dispatch is continuous in the pre-commitment decision $x$, Corollary \ref{rnvsrawithab} applies when the least risk-averse participant's behaviour is sufficiently close to risk-neutrality that the second-stage dispatches under SLP and RASLP are identical. Consequently, the above corollaries can be thought of as risk-averse sensitivity analysis results.

The analysis in the previous sections suggests that increasing the total amount of pre-commitment decreases expected nodal prices. Consequently, a pertinent question is ``does an auxiliary risk market remove the positive relationship between a generator's risk-aversion and its expected payoff?''.

\begin{proposition}\label{expectedriskaverseprofit2}
Let the system be risk-averse with risk measure $\rho$, which has a Kusuoka representation such that $\bar{\beta}:=\int_0^1 \mu^{RN} \beta d\beta$, ${\kappa \in [ 0 , \frac{1}{\bar{\beta}} ]}$, and combine these two quantities by defining $\alpha:=\frac{1}{1+\kappa(1-\bar{\beta})}$. Then, generator $i$'s expected profit is at least $-(1-\alpha)r_{v,i}x_i^\star$.
\end{proposition}

\begin{proof}
Proposition \ref{GeneralRepresentationWithAD} shows that the quantity $\alpha_i:=\frac{1}{\kappa_i-\kappa_i \bar{\beta}_i}$ summarizes the relationship between generation agent $i$'s pre-commitment and its production, since $x_i^\star$ is a $\frac{r_{u,i}+(1-\alpha_i) r_{u,i}}{r_{u,i}+r_{v,i}}$ quantile of the distribution of $X_i^\star(\omega)$. Therefore, $X_i(\omega)^\star \leq x_i^\star$ with probability $\frac{r_{u,i}+(1-\alpha_i) r_{u,i}}{r_{u,i}+r_{v,i}}$. Moreover, it follows from \citep[Corollary 1]{ZPBB} that each generation agent $i$ receives a payoff of at least
$-r_{v,i}x_i^\star$ with probability $\frac{r_{u,i}+(1-\alpha) r_{v,i}}{r_{u,i}+r_{v,i}}$ and receives a payoff of at least $r_{u,i}x_i^\star$ with probability $\frac{\alpha r_{v,i}}{r_{u,i}+r_{v,i}}$. Computing the expected payoff then yields the result.
\end{proof}

The above analysis might appear to suggest that expected cost-recovery is not guaranteed in RASLP. However,
the above analysis does not include payoffs from the auxiliary risk market. Indeed, by comparison with the feasible choice of non-participation in both markets, which has a certain payoff of $0$ under any coherent risk measure, 
it is not too hard to see that risk-averse generators must recover their risk-adjusted costs \textit{in expectation}. However, profits from the auxiliary risk market are derived by assuming risk, unlike
Proposition \ref{expectedriskaverseprofit1}.

\section{Conclusion}
This paper examines the stochastic dispatch problem from the perspective of risk averse participants, and presents a characterization of the impact of pre-commitment on real-time nodal prices, and its interplay with contracts. Furthermore, it provides a characterization of the impact of risk-aversion on pre-commitment, which allows risk-averse equilibria to be elicited according to the optimal solution of a linear program, even with non-polyhedral risk sets. We have demonstrated that risk aversion can be a reason to deviate from a system optimal pre-commitment level, but that a complete risk market would eliminate any incentive to deviate from the system optimal pre-commitment.

\section*{Acknowledgements}
The authors are grateful to Andy Philpott for stimulating discussions on the New Zealand Electricity Market, and several anonymous referees for comments on an earlier version of this paper which improved the quality of the manuscript.  Ryan Cory-Wright is grateful to the Department of Engineering Science at the University of Auckland, the Energy Centre at the University of Auckland and the Energy Education Trust of New Zealand for their financial support.

\begin{appendices}
\section{Omitted Proofs}
\subsection{Proof of Proposition \ref{GeneralRepresentationNoAD}}
\label{sssec:proofnoad}
To show this result, we model an arbitrary generator as a risk-averse newsvendor by using the notation in \cite{MultiProductNewsVendor}, and we convert to the notation used in the main body of this paper ex-post. This choice maintains consistency with the newsvendor literature, because conventional newsvendor models assume the cost of stocking a product is incurred in the first stage, whilst we assume that the cost of stocking a product is incurred in the second stage, and modify our deviation costs accordingly. Our approach can be viewed as a generalization of that taken in Section $5$ of \cite{MultiProductNewsVendor}, as we include the possibility that newsvendors might back-order in the second stage and incur an additional cost for doing so (i.e. ramp up their plant's production at a marginal cost of $c_i+r_{u,i}$), while the analysis conducted by \citet{MultiProductNewsVendor} precludes the possibility that $X^\star_i(\omega)>x^\star_i$.

We require the following terms:
\begin{itemize}
\item $e$ is the marginal emergency order cost.
\item $s$ is the marginal salvage value.
\item $p$ is the marginal sale price.
\item $c$ is the marginal ordering cost.
\item $x$ is the initial order quantity.
\item $D$ is the stochastic demand.
\item $y_{+}=\text{max}(y, 0)$ is the positive component of $y$.
\item $\Pi(x,D)$ is the newsvendor's profit with initial stock $x$ and demand $D$.
\item $\rho$ is a law-invariant coherent risk measure.
\end{itemize}
We
define the newsvendor's profit function as:
\begin{align*}
\Pi(x,D) & =pD-cx+s(x-D)_{+}-e(D-x)_{+},\\
& = (p-e)D+(e-c)x -(e-s)(x-D)_{+},\\
& = (e-c)x+Z_{+},
\end{align*}
where $Z_{+}:=(p-e)D -(e-s)(x-D)_{+}$.

We now invoke Kusuoka's Theorem \citep[see][]{Kusuoka2001} to represent the law-invariant coherent risk measure $\rho$ via the following expression:
\begin{align*}
    \rho[Z]=-\mathbb{E}[Z]+ \kappa \sup_{\mu \in D} \int_{\beta=0}^{\beta=1}\frac{1}{\beta}q_{\beta}[Z] \mu d\beta,
\end{align*}
where $q_{\beta}[Z]=\min_{\eta \in \mathbb{R}}\mathbb{E}[\max((1-\beta)(\eta-Z), \beta(Z-\eta))]= \beta(\mathrm{CVaR}_{\beta}[Z]+\mathbb{E}[Z])$.

The above expression permits a representation of the newsvendor's risk-adjusted profit via the following function:
\begin{align*}
    \rho(\Pi(x,D)) &=-\mathbb{E}[\Pi(x,D)]+\kappa \sup_{\mu \in D} \int_{\beta=0}^{\beta=1}\frac{1}{\beta}q_{\beta}[\Pi(x,D)] \mu d\beta,\\
    &= -(e-c)x -\mathbb{E}[Z_{+}]+\kappa \sup_{\mu \in D} \int_{\beta=0}^{\beta=1}\frac{1}{\beta}q_{\beta}[Z_{+}] \mu d\beta,
\end{align*}
as $(e-c)x$ is invariant and $q_{\beta}[Z+a]=q_{\beta}[Z]$ for nonrandom $a$.
Moving $\mathbb{E}[Z_{+}]$ within the integral and using the substitution \\${q_{\beta}[Z]=\beta(\mathrm{CVaR}_{\beta}[Z]+\mathbb{E}[Z])}$, provides the following expression:
\begin{align*}
    \rho(\Pi(x,D)) = -(e-c)x +\sup_{\mu \in D} \int_{\beta=0}^{\beta=1} \Big( \mathbb{E}[Z_{+}](\kappa \beta -1)\\
    \qquad +\kappa \beta \mathrm{CVaR}_{\beta}[Z_{+}] \Big) \mu d\beta.
\end{align*}

Observe that the $\beta$ quantile of $Z_{+}$ must be lower than in the risk-neutral case. In the risk-neutral case, the optimal choice of $x$ is the $\frac{c-s}{e-s}$ quantile of $D$ (this result is well-known, see \cite{Shapiro}), which corresponds to equality between the $\beta$th quantile of $Z_{+}$ and $(p-s)D$. In the risk-averse case, the $\beta$th quantile of $Z_+$ is (equal or) lower and is therefore equal to ${(p-s)D-(e-s)x}$ for some $x$ and some $D$. This observation allows us to define the partial derivatives of the expectation and $\mathrm{CVaR}$ terms within $\rho(\Pi(x,D)$ as follows:
\begin{align*}
\frac{\partial\mathbb{E}[Z_{+}]}{\partial x}&=-(e-s)\mathbb{P}(x > D), \\
\frac{\partial\mathbb{\mathrm{CVaR}_{\beta}}[Z_{+}]}{\partial x}&=-\frac{\partial}{\partial x}\Big\{(p-s)D-(e-s)x\\
& \qquad -\frac{1}{\beta}\mathbb{E}[(p-s)D-(e-s)x -Z_{+}] \Big\},\\
&=(e-s)-\dfrac{1}{\beta}(e-s)+\dfrac{1}{\beta}(e-s)\mathbb{P}(x > D),\\
&=(e-s)(1-\dfrac{1}{\beta})+\mathbb{P}(x > D)(e-s)\dfrac{1}{\beta}.
\end{align*}
Now, assume that the supremum over the risk set $D$ is uniquely attained at the measure $\hat{\mu}$; then we have the following first-order optimality condition:
\begin{align*}
    \frac{\partial \rho(\Pi(x,D))}{\partial x}&=-(e-c)\\
    &+ \int_{\beta=0}^{\beta=1}\Big( (e-s)( 1+\kappa-\kappa\beta) \mathbb{P}(x > D)\\
    & -\kappa(e-s)(1-\beta) \Big) \hat{\mu} d\beta,\\
    &=-(e-c)+\mathbb{P}(x > D)(e-s)\Big( 1+\kappa-\kappa \big(\int_{\beta=0}^{\beta=1}\beta \hat{\mu} d\beta \big)\Big)\\
    &-\kappa(e-s)(1-(\int_{\beta=0}^{\beta=1} \beta  \hat{\mu} d\beta)).
\end{align*}
Setting this condition to $0$ and re-arranging for $\mathbb{P}(x>D)$ yields:
\begin{align*}
\mathbb{P}(x > D)= \frac{(e-c)+\kappa(e-s)(1-\bar{\beta})}{(e-s)(1+\kappa-\kappa\bar{\beta})},
\end{align*}
where $\bar{\beta}=\int_{0}^{1}\mu^{RN}\beta d\beta$ is the expected value of the risk-averse probabilities with respect to the risk-neutral probabilities.

Netting against $1$ to find $\mathbb{P}(x \leq D)$ then yields:
\begin{align*}
\mathbb{P}(x \leq D)= \frac{(c-s)}{(e-s)(1+\kappa-\kappa\bar{\beta})}.
\end{align*}
To convert to our notation, observe that ${r_u=c-s}$ and ${r_v=e-c}$, giving \\
${r_u+r_v=e-s}$. Therefore, we have that:
\begin{align*}
\mathbb{P}(x \leq D)= \frac{r_u}{(r_u+r_v)(1+\kappa-\kappa\bar{\beta})},
\end{align*}
as required. Note that the corresponding quantile is not necessarily unique because $x^\star$ is attained by solving a linear program (which does not have a unique solution, in general), and if there are multiple optimal choices of $x^\star$ then multiple $X^\star(\omega)$'s may correspond to optimal choices of $x^\star$. \qed

\subsection{Proof of Proposition \ref{GeneralRepresentationWithAD}}
\label{sssec:proofwithad}
To show this result, we invoke the observation made by \cite{ZPBB} that for a given set of second-stage dispatches $X^\star(\omega)$, the system solves a newsvendor problem in order to determine the pre-commitment setpoint which minimizes the term $$\mathbb{E}[r_{u,i}U_i(\omega)+r_{v,i}V_i(\omega)]$$for each generator $i$. Consequently, we use the same notation as in the proof of Proposition \ref{GeneralRepresentationNoAD}. We require $p=0$, as we are considering a system optimization problem and any revenue accrued by a generator is provided by the ISO. Therefore, the system's residual cost with respect to a particular generator's pre-commitment decision, $\Pi(x,D)$, is defined by the following expression:
\begin{align*}
\Pi(x,D) & =-cx+s(x-D)_{+}-e(D-x)_{+},\\
& = -eD+(e-c)x -(e-s)(x-D)_{+},\\
& = (e-c)x+Z_{+}.
\end{align*}
By following the steps outlined in Appendix \ref{sssec:proofnoad}, we obtain the following risk-adjusted profit function:
\begin{align*}
    \rho(\Pi(x,D)) = -(e-c)x -\mathbb{E}[Z_{+}]+\kappa \sup_{\mu \in D} \int_{\beta=0}^{\beta=1}\dfrac{1}{\beta}q_{\beta}[Z_{+}] \mu d\beta.
\end{align*}
Now observe that since $p=0$, the risk-neutral critical fractile becomes $-eD$. Since $s(x-D)-ex$ gives a lower system cost than $-eD$ we therefore have that the critical fractile within each $\mathrm{CVaR}$ term becomes equal to $-eD$. This situation is similar to Section $5.3$ of \cite{MultiProductNewsVendor}, although we include emergency holding costs. Consequently, the partial derivatives of the terms which constitute $\rho$ become:
\begin{align*}
\frac{\partial\mathbb{E}[Z_{+}]}{\partial x}&=-(e-s)\mathbb{P}(x > D), \\
\frac{\partial\mathbb{\mathrm{CVaR}_{\beta}}[Z_{+}]}{\partial x}&=-\frac{\partial}{\partial x}\Big\{-eD-\dfrac{1}{\beta}\mathbb{E}[-eD -Z_{+}] \Big\}=\dfrac{1}{\beta}(e-s)\mathbb{P}(x > D).
\end{align*}
Now assume that $\mu=\hat{\mu}$ is the unique optimal pdf. Then, we have the following first-order condition:
\begin{align*}
    \frac{\partial \rho(\Pi(x,D))}{\partial x}&=-(e-c)+\int_{\beta=0}^{\beta=1}\Big( (e-s)( 1+\kappa-\kappa\beta) \mathbb{P}(x > D)\Big) \hat{\mu} d\beta,\\ &=-(e-c)+\mathbb{P}(x > D)(e-s)\Big( 1+\kappa-\kappa \int_{\beta=0}^{\beta=1}\beta \hat{\mu} d\beta\Big).
\end{align*}
Setting this condition to $0$ and re-arranging for $\mathbb{P}(x>D)$ yields:
\begin{align*}
\mathbb{P}(x > D)= \frac{(e-c)}{(e-s)(1+\kappa-\kappa\bar{\beta})},
\end{align*}
where $\bar{\beta}=\int_{0}^{1}\mu^{RN}\beta d\beta$ is the expected value of the risk-averse probabilities with respect to the risk-neutral probabilities.\\
Netting against $1$ to find $\mathbb{P}(x \leq D)$ then yields:
\begin{align*}
\mathbb{P}(x \leq D)= \frac{(c-s)+(e-s)(\kappa-\kappa\bar{\beta})}{(e-s)(1+\kappa-\kappa\bar{\beta})}.
\end{align*}
To convert to our notation, observe that ${r_u=c-s}$ and ${r_v=e-c}$, giving \\
${r_u+r_v=e-s}$. Therefore, we have that:
\begin{align*}
\mathbb{P}(x \leq D)= \frac{r_u+(r_u+r_v)(\kappa-\kappa\bar{\beta})}{(r_u+r_v)(1+\kappa-\kappa\bar{\beta})},
\end{align*}
as required.
Note that the corresponding quantile is not necessarily unique. \qed
\end{appendices}

\end{document}